\documentclass[10pt]{amsart}
\usepackage{amsmath,amscd}
\usepackage{amsbsy}
\usepackage{amssymb}
\usepackage{amscd,amsthm}
\usepackage[all,cmtip]{xy}
\usepackage[russian,english]{babel}
\usepackage[utf8]{inputenc}
\usepackage{xcolor}

\usepackage{hyperref}

\renewcommand{\epsilon}{\varepsilon}

\newtheorem{theorem}{Theorem}

\renewcommand{\ell}{x}

\newtheorem{thm}{Theorem}\numberwithin{thm}{section}
\newtheorem{lem}[thm]{Lemma}
\newtheorem{prop}[thm]{Proposition}
\newtheorem{cor}[thm]{Corollary}

\newtheorem*{con2}{Conjecture}

\begin{document}
\begin{center}
\huge{The Diophantine equation $P(x)=\overset{r}{\underset{i=1}{\prod}}H_{n_i}$ }\\[1cm] 
\end{center}
\begin{center}

\large{Sa$\mathrm{\check{s}}$a Novakovi$\mathrm{\acute{c}}$}\\[0,5cm]
{\small January 2026}\\[0,5cm]
\end{center}
{\small \textbf{Abstract}. 
Naciri proved that for any integer $k\geq2$, the Brocard--Ramanujan equation $n!+1=x^2$ has only finitely many integer solutions, assuming $x\pm1$ is a $k$-free integer or a prime power. In the present paper we prove similar statements for equations of the form $P(x)=\prod_{i=1}^rH_{n_i}$, where $P(x)$ is a polynomial and $H_{n_i}$ are divisible sequences.}
\begin{center}
	\tableofcontents
\end{center}
\section{Introduction}
The theory of Diophantine equations has a long and rich history and has attracted the attention of many mathematicians. In particular, the study of diophantine equations involving factorials have been studied extensively. For example Brocard \cite{BR}, and independently Ramanujan \cite{RA}, asked to find all integer solutions for $n!=x^2-1$. It is still an open problem, known as Brocard's problem, and it is believed that the equation has only three solutions $(x,n)=(5,4), (11,5)$ and $(71,7)$. Overholt \cite{O} observed that a weak form of Szpiro's conjecture implies that Brocard's equation has finitely many integer solutions. Quite recently, Naciri \cite{NA} showed that  $n!+1=x^2$ has only finitely many integer solutions, assuming $x\pm1$ is a $k$-free integer or a prime power. Some further examples of similar equations are:
\begin{itemize}
	\item[1)] $n!=x^k\pm y^k$ and $n!\pm m!=x^k$, see \cite{EO}.
	\item[2)] $\phi(x)=n!$, where $\phi$ is the Euler totient function \cite{FL}.
	\item[3)] $p(x)=m!$, where $p(x)\in\mathbb{Z}[x]$ \cite{L}.
	\item[4)] $\alpha\,m_1!_{S_1}\cdots m_r!_{S_r}=f(n!)$, where $f$ is an arithmetic function and $m_i!_{S_i}$ are certain Bhargava factorials \cite{BN}.
\end{itemize}
For the equations 1) and 4), it was shown that the number of integer solutions is finite. The equation in 3) has finitely many integer solutions, provided the ABC conjecture holds, and 2) does have infinitely many solutions. There are a lot of more diophantine equations involving factorials and polynomials that have been studied and we refer the interested reader to \cite{BN}, \cite{NO} and the references therein. There are also variants of the Brocard--Ramanujan equation that have been studied intensively in the literature. We do not want to give a complete list here and refer for instance to \cite{KL} and \cite{NA1}. Instead, we focus on the paper of Berend and Harmse \cite{BH1}, where the authors studied diophantine equations of the form $P(x)=H_n$, where $H_n$ are certain divisible sequences, namely $H_n=n!$ or $H_n=p_n\#$ or $H_n=[1,...,n]$. In \cite{NO} and \cite{NOO} the author considered the cases where $H_n=bA^nn!$ and $H_n=bA^nn!!$ and showed that under the ABC conjecture the equations $P(x)=\prod_{i=1}^rA_i^{n_i}n_i!$ and $P(x)=\prod_{i=1}^rA_i^{n_i}n_i!!$ have only finitely many integer solutions $(x,n_1,...,n_r)$. So from this perspective, it is reasonable to study, more generally, equations of the form $P(x)=\prod_{i=1}^rH_{n_i}$, where $H_{n_i}$ are arbitrary divisible sequences. In the present paper, we focus on equations of the form $P(x)=\prod_{i=1}^rH_{n_i}$, where the $H_{n_i}$ are allowed to be $A_i^{n_i}n_i!$ or $A_i^{n_i}n_i!!$ or $A_i^{p_i}p_{n_i}\#$ or $A_i^{n_i}p_i\#$ or $A_i^{n_i}n_i\#$ or $A_i^{n_i}[1,...,n_i]$ with fixed positive integers $A_i$. There are certainly other divisible sequences $H_{n_i}$ to consider, but we study only on the mentioned ones, because the arguments and techniques used in the present paper at least work for these cases. 
\begin{theorem}
	Let $P(x)\in \mathbb{Z}[x]$ be an irreducible polynomial over $\mathbb{Q}$ of degree $d\geq 2$. Let $H_{n_i}$ be either $A_i^{n_i}[1,...,n_i], A_i^{n_i}n_i!, A_i^{n_i}n_i!!$ or $A_i^{p_i}p_{n_i}\#$ or $A_i^{n_i}p_i\#$ or $A_i^{n_i}n_i\#$. Here $A_i$ are fixed positive integers. Then the diophantine equation $P(x)=\prod_{i=1}^rH_{n_i}$ has only finitely many integer solutions $(x,n_1,...,n_r)$.
	\end{theorem}
\noindent
Theorem 1 elaborates on the case of irreducible polynomials. If $P(x)$ is arbitrary, the problem gets more involved and we are only able to prove finiteness of the set of integer solutions for certain divisible sequences and under the ABC conjecture. So we recall its statement. For a non-zero integer $a$, let $N(a)$ be the \emph{algebraic radical}, namely $N(a)=\prod_{p|a}{p}$.
\begin{con2}[ABC conjecture]
	For any $\epsilon >0$ there is a constant $K(\epsilon)$ depending only on $\epsilon$ such that whenever $A,B$ and $C$  are three coprime and non-zero integers with $A+B=C$, then 
	\begin{eqnarray*}
		\mathrm{max}\{|A|,|B|,|C|\}<K(\epsilon)N(ABC)^{1+\epsilon}
	\end{eqnarray*}
	holds.
\end{con2} 
\begin{theorem}
	Let $P(x)\in\mathbb{Q}[x]$ be a polynomial of degree $d\geq 2$ which is not monomial and has at least two distinct roots. Let $A_i$ be fixed positive integers and let $H_{n_i}$ be 
	\begin{itemize}
		\item[(i)] $A_i^{n_i}n_i!$, $A_i^{n_i}n_i!!$ or 
		\item[(ii)] $A_i^{n_i}n_i\#$ or $A_i^{n_i}[1,...,n_i]$ with $A_i>1$,
	\end{itemize}
	then the ABC conjecture implies that the diophantine equation $P(x)=\prod_{i=1}^rH_{n_i}$ has only finitely many integer solutions $(x,n_1,...,n_r)$.
	\end{theorem}
\noindent
In \cite{DAB} Dabrowski studied the equation $n!+A=x^2$ and showed, among others, that if $A$ is a square, the weak form of Szpiro's conjecture implies that the equation has finitely many integer solutions. We generalize this result a little bit by allowing the left hand side to be a product of some divisible sequences. Let us recall Szpiro's conjecture.
\begin{con2}[Weak form of Szpiro's conjecture]
	There exists some constant $s>0$ such that for mutually prime integers $A,B$ and $C$  with $A+B=C$ the inequality 
	\begin{eqnarray*}
		|ABC|<N(ABC)^{s}
\end{eqnarray*}
\end{con2}
\begin{theorem}
	Let $u\geq 1$ be a positive integer. Then the weak form of Szpiro's conjecture implies that $\prod_{i=1}^rA_i^{n_i}n_i!=x^2-u^2$ has finitely many integer solutions.
\end{theorem}
Throughout this paper we denote by $\omega(x)$ the number of distinct primes in the prime facotization of a positive integer $x$ and by $K(x)$ the maximum exponent in the prime factorization. For any pair $k,l\geq 2$, let us define
$$
\mathcal{F}_k=\{x\in \mathbb{N}\mid K(x)<k\} \quad \textnormal{and}\quad \mathcal{P}_l=\{x\in \mathbb{N}\mid \omega(x)<l\}.
$$
We fix an integer $a\in\mathbb{Z}$ and a polynomial $R(x_1,...,x_m)\in \mathbb{Z}[x_1,...,x_m]$ and define the following set
$$
S_{a,R}=\{x_1\cdots x_m\cdot R(x_1,...,x_m)+a\mid x_1,...,x_m\in \mathcal{F}_k\mathcal{P}_l\}.
$$
\begin{theorem}
	Fix positive integers $A_1,...,A_r$. Let $P(x)=(x-a)^e\prod_{i=1}^sP_i(x)^{e_i}$, where $P_i(x)$ are distinct irreducible polynomials. Then the diophantine equation $P(x)=\prod_{i=1}^sA_i^{n_i}n_i!\cdot \prod_{i=s+1}^rA_i^{n_i}n_i!!$ has finitely many solutions $(n_1,...,n_r,x+a)$, where $x\in \mathcal{F}_k$.
\end{theorem}
\noindent
In a similar way, one can show:
\begin{theorem}
		Let $P(x)=(x-a)^e\prod_{i=1}^sP_i(x)^{e_i}$, where $P_i(x)$ are distinct irreducible polynomials. Then the diophantine equation $P(x)=\prod_{i=1}^sn_i!\cdot \prod_{i=s+1}^rn_i!!$ has finitely many solutions $(n_1,...,n_r,x)$ where $x\in S_{a,R}$.
\end{theorem}
\begin{theorem}
	Let $d>r$. Then the equation $x^d=\prod_{i=1}^rH_{n_i}$ has finitely many integer solutions, where $H_{n_i}$ is allowed to be any of the sequences $A_i^{n_i}n_i!, A_i^{n_i}n_i!!$ or $A_i^{n_i}[1,...,n_i]$ or $A_i^{n_i}n_i\#$ or $A_i^{p_{n_i}}p_{n_i}\#$. 
\end{theorem}
\noindent
In the case $d\leq r$ it seems that we have no general and definite answer. Instead, whether there are finitely many or infinitely many solutions depends heavily on the concrete factors in the product $\prod_{i=1}^rH_{n_i}$. For example, if $d=2$, the equations $x^2=A^nn!\cdot B^mm!$ and $x^2=A^nn!!\cdot B^mm!!$ are treated in the papers \cite{NOO} and \cite{NO}. In both cases there are infinitely many integer solutions and they can be constructed. If we consider $x^2=n!\cdot m!!$, then $m=2l, n=l$ with $l$ even gives $x^2=l!2^ll!$ infinitely many solutions. Obviously, the equation $x^2=n\#\cdot m\#$ has infinitely solutions as well. Just set $n=m$. If we consider equations of the form $x^2=n!\cdot m\#\cdot l!$, then by setting $m=1$, we conclude that there are again infinitely many solutions. We believe that finitely many solutions occur for $x^2=n!\cdot n\#$ or $x^2=n!!\cdot [1,...n]$, but at the moment we were not able to give an argument.

\section{Preliminary Results}
In this section we collect certain facts that we shall apply in the next sections. We state some classical number theoretic results. Let $\pi(n)$ be the number of primes less than or equal to $n$. Chebyshev's theorem yields a refinement of the prime number theorem and states
$$
\frac{n}{\textnormal{ln}(n)}\leq \pi(n)\leq \frac{3}{2}\frac{n}{\textnormal{ln}(n)},
$$
for all $n\geq 2$. Furthermore, Stirling's formula states that $n!$ asymptotically behaves like $n^n\textnormal{exp}(-n)\sqrt{2\pi n}$ as $n\rightarrow \infty$. Moreover, for all $n\geq 1$, we have
$$
\left(\frac{n}{e}\right )^n\leq n!\leq n^n.
$$
For a prime $p$, let $\nu_p(n)$ denote the largest power of $p$ dividing $n$. Legendre's formula states
$$
\nu_p(n!)=\sum_{i=1}^{\infty}\lfloor\frac{n}{p^i} \rfloor,
$$
where $\lfloor x\rfloor$ is the floor function. This yields
$$
\nu_p(n!)\leq \sum_{i=1}^{\infty}\frac{n}{p^i}\leq \frac{n}{p-1}.
$$
All these facts can be found for instance in \cite{HW}. Fix positive integers $s$ and $A$. It is easy to conclude that
$$
\nu_p(A^nn!^s)=\nu_p(A^n)+\nu_p(n!^s)=n\nu_p(A)+s\nu_p(n!)\leq n\nu_p(A)+\frac{ns}{p-1}.
$$

As mentioned in the introduction, in the prime factorization of a positive integer 
$$
x=\prod_{i\in I}p_i^{\alpha_i}
$$
we denote $\omega(x)=|I|$ and $K(x)=\underset{i\in I}{\textnormal{max}}\ \alpha_i$.
\begin{lem}
	Fix positive integers $s$ and $A$. Let $x$ and $n$ be positive integers and assume $x$ divides $A^nn!^s$. Then
	$$
	x\leq A^{K(x)\omega(x)}\cdot \textnormal{exp}\left(\frac{3nsK(x)}{2} \right)
	$$
\end{lem}
\begin{proof}
	Let $x=\underset{i\in I}{\prod}p_i^{\alpha_i}$ be the prime factorization of $x$. By definition, $\alpha_i\leq K(x)$ for all $i\in I$. Therefore 
	$$
	x\leq \underset{i\in I}{\prod}p_i^{K(x)}.
	$$
	And since $x$ divides $A^nn!^s$, we conclude $p_i\leq An^s$. Note that even $p_i\leq An$. The number of distinct prime factors $|I|$ is bounded by $\pi(n)$. Hence
	$$
	x\leq (An^s)^{K(x)|I|}
	\leq A^{K(x)|I|}n^{sK(x)\pi(n)}=A^{K(x)\omega(x)}\cdot \textnormal{exp}\left(\frac{3nsK(x)}{2} \right).
	$$
\end{proof}
\noindent
In particular, if $n\geq A$ and $x\mid A^nn!$, then $A^{K(x)\omega(x)}\leq n^{K(x)|I|}$ and the proof of Lemma 2.1 yields
$$
x\leq A^{K(x)\omega(x)}\cdot n^{sK(x)\pi(n)}\leq  n^{2sK(x)\pi(n)}=\textnormal{exp}\left(3nsK(x)\right).
$$
\begin{lem}
	Fix positive integers $s$ and $A$. Let $x$ and $n$ be positive integers and assume $x$ divides $A^nn!^s$. Then
	$$
	x\leq \left(a\omega(x)+1\right)^{n\omega(x)(\beta + s)}(An^s)^{n(\beta\omega(x)+\frac{s}{a})}
	$$
	for any arbitrary positive integer $a$.
\end{lem}
\begin{proof}
	Let $x=\underset{i\in I}{\prod}p_i^{\alpha_i}$ be the prime factorization of $x$ and assume $x$ divides $A^nn!^s$. Notice that $|I|=\omega(x)$, $p_i\leq An^s$ for all $i\in I$, and $\alpha_i\leq n\nu_{p_i}(A)+\frac{sn}{p_i-1}$ for all $i\in I$. If $\beta$ denotes the maximal exponent in the prime factorization of $A$, we have $\alpha_i\leq n\beta+\frac{sn}{p_i-1}$. Now, for a positive integer $a$, define the set
	$$
	J=\{i\in I\mid p_i\leq a\omega(x)+1\}.
	$$
	We then have
	$$
	\underset{i\in J}{\sum}\alpha_i\leq \underset{i\in I}{\sum}(n\beta +\frac{sn}{p_i-1})\leq n\beta\omega(x)+sn\omega(x)=n\omega(x)(\beta+s)
	$$
	and
	$$
	\underset{i\in I\setminus J}{\sum}\alpha_i\leq \underset{i\in I\setminus J}{\sum}(n\beta+\frac{sn}{p_i-1})\leq n\beta\omega(x)+\underset{i\in I\setminus J}{\sum}\frac{sn}{a\omega(x)}\leq n\beta\omega(x)+\frac{sn}{a}.
	$$
	Therefore,
	$$
	x=\underset{i\in I}{\prod}p_i^{\alpha_i}=\underset{i\in J}{\prod}p_i^{\alpha_i}\underset{i\in I\setminus J}{\prod}p_i^{\alpha_i}\leq \left(a\omega(x)+1\right)^{n\omega(x)(\beta + s)}(An^s)^{n\beta\omega(x)+\frac{sn}{a}}.
	$$
\end{proof}
\noindent
If $A=1$, the inequality in Lemma 2.2 simplifies to
$$
x\leq \left(a\omega(x)+1\right)^{n\omega(x)s}(n^s)^{\frac{ns}{a}},
$$
bacause $\beta=0$. 
\noindent
We also need the next result, which is essentially \cite{WT}, Theorem 5.5. At this point, we want to mention that the statement presented in \emph{loc. cit.} needs a minor correction, namely the exponent $1+\epsilon$ must be added. In \cite{WT} this exponent is missing. We formulate a more general result, including the case stated in \cite{WT}, Theorem 5.5 and taking into account the corresponding minor correction.
\begin{thm}
	Let $P(x)\in\mathbb{Q}[x]$ be a polynomial of degree $d\geq 2$ which is not monomial and has at least two distinct roots. Let $F(n_1,...,n_r)$ be a function and assume there exists an $\epsilon>0$, such that $N(F(n_1,...n_r))^{1+\epsilon}=o(F(n_1,...n_r))$ as $(n_1,...,n_r)\rightarrow \infty$. Then the ABC-conjecture implies that $P(x)=F(n_1,...,n_r)$ has finitely many integer solutions.
\end{thm}
\begin{proof}
In  \cite{L}, it is shown that $P(x)=n!$ has finitely many integer solutions, provided the ABC conjecture holds. In fact, the crucial step in the proof is to show that for a suitable $\epsilon >0$ one has $N(n!)^{1+\epsilon}=o(n!)$ as $n\rightarrow \infty$. In the steps before, the ABC conjecture was essentially applied to obtain an inequality of the form $n!\leq C\cdot N(n!)^{1+\epsilon}$, where $C$ is a suitable positive constant. Notice that this argument can be generalized to functions $F(n_1,...n_r)$. So, if there exists an $\epsilon>0$, such that $N(F(n_1,...n_r))^{1+\epsilon}=o(F(n_1,...n_r))$ as $(n_1,...,n_r)\rightarrow \infty$, then the ABC conjecture implies that $P(x)=F(n_1,...,n_r)$ has finitely many integer solutions. 
\end{proof}
\noindent
Furthermore, we also need the next assertion which essentially follows from the Chebotarev density theorem and can be found in \cite{J}, p.138-139. 
\begin{thm}
	Let $P(x)\in \mathbb{Z}[x]$ be irreducible over $\mathbb{Q}$ of degree $d\geq 2$. Then there are infinitely many primes $p$ such that $P(x)\equiv 0\ (\textnormal{mod}\ p)$ has no solutions.
\end{thm}
\noindent
The following facts are well known.
\begin{prop}
	Let $A$ be a positive integer. Then
	\begin{itemize}
		\item[(i)] $N(A^nn!)\leq N(A)\cdot 4^n$,
		\item[(ii)] $N(A^nn!!)\leq N(A)\cdot 4^n$,
		\item[(iii)] $N(A^nn\#)\leq N(A)\cdot n\#$,
		\item[(iv)] $N(A^n[1,...,n])\leq N(A)\cdot n\#$.
		\end{itemize}
	\end{prop}

\begin{cor}
	Let $A_1,...,A_c$ be fixed positive integers. Then for suitable $\epsilon>0$ the following hold:
	\begin{itemize}
		\item[(i)] Let $F(n_1,..,n_c)=\prod_{i=1}^cA_i^{n_i}n_i!$, then $\frac{N(F(n_1,...,n_c))^{1+\epsilon}}{F(n_1,...,n_c)}\rightarrow 0$ as $(n_1,...,n_c)\rightarrow \infty$.
		\item[(ii)] Let $F(n_1,..,n_c)=\prod_{i=1}^cA_i^{n_i}n_i!!$, then $\frac{N(F(n_1,...,n_c))^{1+\epsilon}}{F(n_1,...,n_c)}\rightarrow 0$ as $(n_1,...,n_c)\rightarrow \infty$. 
		\item[(iii)] Let $F(n_1,..,n_c)=\prod_{i=1}^cA_i^{n_i}$ and assume that at least one $A_i>1$, then $\frac{N(F(n_1,...,n_c))^{1+\epsilon}}{F(n_1,...,n_c)}\rightarrow 0$ as $(n_1,...,n_c)\rightarrow \infty$. 
			\item[(iv)] Let $F(n_1,..,n_c)=\prod_{i=1}^cA_i^{n_i}{n_i}\#$ and assume $A_i>1$ for all $i=1,...,c$, then $\frac{N(F(n_1,...,n_c))^{1+\epsilon}}{F(n_1,...,n_c)}\rightarrow 0$ as $(n_1,...,n_c)\rightarrow \infty$. 
		\item[(v)] Let $F(n_1,..,n_c)=\prod_{i=1}^cA_i^{n_i}[1,...,n_i]$ and assume $A_i>1$ for all $i=1,...,c$, then $\frac{N(F(n_1,...,n_c))^{1+\epsilon}}{F(n_1,...,n_c)}\rightarrow 0$ as
		 $(n_1,...,n_c)\rightarrow \infty$.
	\end{itemize}
\end{cor}	
\begin{proof}
We give the proofs for the cases (iv) and (v), because the other cases follow easily from Proposition 2.5 (i) and (ii) using Stirling's approximation. We first give an argument for (iv). To make the argument as clear as possible, we treat the case $c=2$. So let $F(n_1,n_2)=A_1^{n_1}n_1\#\cdot A_2^{n_2}n_2\#$. With Proposition 2.5 (iii), we obtain
$$
N(F(n_1,n_2))^{1+\epsilon}\leq (N(A_1)\cdot N(A_2))^{1+\epsilon}(n_1\#\cdot n_2\#)^{1+\epsilon}
$$
and therefore
$$
0\leq \frac{F(n_1,n_2)^{1+\epsilon}}{F(n_1,n_2)}\leq (N(A_1)\cdot N(A_2))^{1+\epsilon}\cdot \frac{(n_1\#)^{\epsilon}}{A_1^{n_1}}\cdot \frac{(n_2\#)^{\epsilon}}{A_2^{n_2}}.
$$
Using the asymptotics relation $n\#\sim e^n$ and the fact that we can choose an $\epsilon>0$, such that $e^{\epsilon}<\mathrm{min}\{A_1,A_2\}$, we finally conclude
$$
 \frac{(n_1\#)^{\epsilon}}{A_1^{n_1}}\cdot \frac{(n_2\#)^{\epsilon}}{A_2^{n_2}}\rightarrow 0,
$$
as $(n_1,n_2)\rightarrow \infty$. This proves (iv). To show (v), we use Proposition 2.5 (iv) and conclude as above by exploiting the fact that
$$
\frac{n\#^{1+\epsilon}}{A^n[1,...,n]}=\frac{\underset{p\leq n}{\prod} p^{1+\epsilon}}{A^n\underset{p\leq n}{\prod} p^{\lfloor\mathrm{log}_p(n)\rfloor}}\leq \frac{n\#^{1+\epsilon}}{A^nn\#}=\frac{n\#^{\epsilon}}{A^n}.
$$
\end{proof}

\section{Proof of Theorem 1}
\noindent
In fact this is a direct consequence of Theorem 2.4. Let $H_n=A^nn!$, $H_n=A^nn!!$ or $H_n=A^n[1,...,n]$ or $H_n=A^nn\#$ and assume $P(x)=H_n$ has infinitely many integer solutions. In \cite{BH1}, the authors showed that $P(x)\equiv 0\ (\textnormal{mod}\ m)$ for every prime power $m=p^k$, or, equivalently, $P$ has a root in $\mathbb{Z}_p$ for every prime $p$ or, in the case $H_n=A^nn\#$, that $P(x)\equiv 0\ (\textnormal{mod}\ m)$ must be solvable for every square free integer $m$. It is easy to see that the same argument also applies for the equation $P(x)=\prod_{i=1}^rH_{n_i}$, where $H_{n_i}$ are the divisible sequences listed in the statement of Theorem 1. The finiteness of the set of integer solutions then follows directly from Theorem 2.4.
\section{Proof of Theorem 2}
We have to show that $N(\prod_{i=1}^rH_{n_i})^{1+\epsilon}=o(\prod_{i=1}^rH_{n_i})$ for a suitable $\epsilon>0$. Then the finiteness of the set of integer solutions under the ABC conjecture follows from Theorem 2.3. In fact, we have
$$
N(\prod_{i=1}^rH_{n_i})^{1+\epsilon}\leq \prod_{i=1}^{r}N(H_{n_i})^{1+\epsilon}
$$
and our argument reduces to show that for a suitable $\epsilon>0$
$$
\frac{N(H_{n_i})^{1+\epsilon}}{H_{n_i}}\rightarrow 0,
$$
as $n_i\rightarrow \infty$. But this follows directly from Corollary 2.6.
\section{Proof of Theorem 3}
\noindent
Assume $(n_1,...,n_r)$ to be large enough to make $u^2$ divide $\prod_{i=1}^rA_i^{n_i}n_i!$. Then $u$ also divides $x$. Put $x=ux_0$ and without loss of generality we may assume that $x_0$ is odd. Hence $\prod_{i=1}^rA_i^{n_i}n_i!=d^2(x_0^2-1)$. Now we use the weak form of Szpiro's conjecture with $A=1, B=\frac{1}{2}(x_0-1)$ and $C=\frac{1}{2}(x_0+1)$. We then obtain
$$
\frac{d^{-2}}{4}\prod_{i=1}^r(\frac{n_i}{e})^{n_i}\leq \frac{x_0^2-1}{4}\leq N(\frac{x_0^2-1}{4})^s=N(\frac{1}{4}\prod_{i=1}^rA_i^{n_i}n_i!)\leq \prod_{i=1}^rN(A_i)^s\cdot 4^{s(n_1+\cdots+n_r)}.
$$
By symmetry, we may assume $n_1\leq \cdots \leq n_r$. This yields
$$
\frac{d^{-2}}{4}(\frac{n_r}{e})^{n_r}\leq \prod_{i=1}^rN(A_i)^s\cdot 4^{srn_r}.
$$
But this holds only for finitely many $n_r$. Hence there are finitely many integer solutions $(n_1,...,n_r,x)$. 
\section{Proof of Theorems 4 and 5}
\noindent
We first prove Theorem 4. Let $x\in \mathcal{F}_k$ and let $d$ be the degree of $P$. Then
$$
P(x+a)=x^e\prod_{i=1}^tP_i(x+a)^{e_i}\leq M\cdot x^d,
$$
where $M$ is the sum of the absolute values of the coefficients of the polynomial $P$. Let $(n_1,...,n_r,x+a)$ be a solution. Since 
$$
P(x+a)=\prod_{i=1}^sA_i^{n_i}n_i!\cdot \prod_{i=s+1}^rA_i^{n_i}n_i!!,
$$
we see that $x$ divides the right hand side. Without loss of generality, we may assume $n_1\leq n_2\leq\cdots \leq n_r$. Assuming this, we obtain that $\prod_{i=1}^sA_i^{n_i}n_i!\cdot \prod_{i=s+1}^rA_i^{n_i}n_i!!$ divides $(\prod_{i=1}^rA_i)^{n_r}n_r!^r$. Therefore, 
$$
x\mid C^{n_r}n_r!^r,
$$
where $C=\prod_{i=1}^rA_i$. We now apply Lemma 2.1 and the comment immediately afterwards by assuming $n_r\geq C$. This gives 
$$
A^{n_r}n_r!!\leq \prod_{i=1}^sA_i^{n_i}n_i!\cdot \prod_{i=s+1}^rA_i^{n_i}n_i!!=P(x+a)\leq M\cdot x^d\leq M\cdot \textnormal{exp}(3dn_r\cdot rk).
$$
Now the Stirling formula for $n!!$ implies that there are only finitely many $n_r$. Therefore, our diophantine equation has only finitely many integer solutions $(n_1,...,n_r,x+a)$, with $x\in \mathcal{F}_k$. This completes the proof of Theorem 3.\\
\noindent
We now prove Theorem 5. Let $x\in S_{a,R}$ and assume $x$ is a solution. Then
$$
P(x)=P(x_1\cdots x_mR(x_1,...,x_m)+a)=(x_1\cdots x_mR(x_1,...,x_m))^e\cdot \prod_{i=1}^sP_i(x+a)^{e_i}
$$
and we see that $x_1\cdots x_m$ divides $\prod_{i=1}^sn_i!\cdot \prod_{i=s+1}^rn_i!!$. Now define 
$$
Q(X_1,...,X_m):=P(X_1\cdots X_mR(X_1,...,X_m)+a)
$$
and rewrite $Q$ as 
$$
Q=\underset{(i_1,...,i_m)\in L}{\sum}a_{i_1,...,i_m}X_1^{i_1}\cdots X_m^{i_m}.
$$ 
The total degree of $Q$ is given by
$$
d=\underset{(i_1,...,i_m)\in L}{\textnormal{max}}i_1+\cdots +i_m.
$$
Furthermore, let
$$
M=\underset{(i_1,...,i_m)\in L}{\sum}|a_{i_1,...,i_m}|
$$
and notice that for all positive integers $x_1,...,x_m$, one has
$$
Q(x_1,...,x_m)\leq M\cdot (\underset{1\leq i\leq m}{\textnormal{max}}x_i)^d.
$$
Now let $x=x_1\cdots x_mR(x_1,...,x_m)+a$, $n_1,...,n_r$ be a solution. Then $Q(x_1,...,x_m)=\prod_{i=1}^sn_i!\cdot \prod_{i=s+1}^rn_i!!$ and therefore
$$
\prod_{i=1}^sn_i!\cdot \prod_{i=s+1}^rn_i!!\leq M\cdot x_{i_0}^d,
$$
where $\underset{1\leq i\leq m}{\textnormal{max}}x_i=x_{i_0}$. By definition, $x_{i_0}=yz$, with $y\in\mathcal{F}_k$ and $z\in\mathcal{P}_l$. Because $x_1\cdots x_m$ divides $\prod_{i=1}^sn_i!\cdot \prod_{i=s+1}^rn_i!!$, the integer $x_{i_0}$ divides the product $\prod_{i=1}^sn_i!\cdot \prod_{i=s+1}^rn_i!!$ as well. Without loss of generality, we assume $n_1\leq n_2\leq \cdots \leq n_r$. But then $x_{i_0}$ divides $n_r!^r$ and hence $y\mid 
n_r!^r$ and $z\mid n_r!^r$. Now we apply Lemma 2.1 and Lemma 2.2. This yields
$$
y\leq \textnormal{exp}(3n_r\cdot rk)
$$
and
$$
z\leq (al+1)^{n_rlr}n^{n_r\cdot\frac{r^2}{a}}.
$$
Combining the latter inequalities gives
$$
n_r!!\leq \prod_{i=1}^sn_i!\cdot \prod_{i=s+1}^rn_i!!\leq M\cdot\textnormal{exp}(3n_r\cdot rkd)\cdot (al+1)^{n_rlrd}\cdot n_r^{n_r\frac{r^2d}{a}}.
$$
Again, using the Stirling approximation for $n!!$ with $a=4r^2d$ shows that there are only finitely many $n_r$. And since $n_1\leq n_2\leq\cdots \leq n_r$, it follows that there can be only finitely many solutions $(n_1,...,n_r,x)$, with $x\in S_{a,R}$. 
\section{Proof of Theorem 6}
\noindent
By assumption, it is $d>r$. 
Let $n_j=\textnormal{max}\{n_1,...,n_r\}$. Futhermore, assume $n_j>16\cdot \textnormal{max}\{A_1,...,A_r\}$. There are two cases to distinguish.\\
\noindent
\emph{the case} $H_{n_j}$ is not $A^{n_j}n_j!!$: In this case, $H_{n_j}$ is either $A_i^{n_i}n_i!, A_i^{n_i}[1,...,n_i]$ or $A_i^{n_i}n_i\#$ or $A_i^{p_{n_i}}p_{n_i}\#$. Notice that $[1,...,n_j]=\underset{p\leq n_j}{\prod}p^{[\textnormal{log}_p(n_j)]}$. Hence, we can find a prime $q\in (n_j/2,n_j)$. For this prime $q$, we have $\nu_q(\prod_{i=1}^rH_{n_i})\leq r$, whereas $\nu_q(x^d)\geq d$. This gives a contradiction.\\
\noindent
\emph{the case} $H_{n_j}=A^{n_j}n_j!!$: If $n_j$ is odd, we choose a prime $p\in (n_j/2,n_j)$ and see that $\nu_p(\prod_{i=1}^rH_{n_i})\leq r$, whereas $\nu_p(x^d)\geq d$. If $n_j$ is even, it is either $n_j/2\geq \underset{i\neq j}{\textnormal{max}}\{n_i\}$ or $n_j/2< \underset{i\neq j}{\textnormal{max}}\{n_i\}$. If  $n_j/2\geq \underset{i\neq j}{\textnormal{max}}\{n_i\}$, we can pick a prime $p\in (n_j/4,n_j/2)$ and conclude that $\nu_p(\prod_{i=1}^rH_{n_i})\leq r$. If $n_j/2< \underset{i\neq j}{\textnormal{max}}\{n_i\}$, then let $n_{i_0}=\textnormal{max}\{n_1,...,n_j/2,...,n_r\}$. Now pick a prime $p'\in(n_{i_0}/2,n_{i_0})$ and conclude $\nu_{p'}(\prod_{i=1}^rH_{n_i})\leq r$, whereas $\nu_{p'}(x^d)\geq d$. This completes the proof and shows that there can be only finitely many integer solution.

\vspace{0.3cm}
\noindent
{\tiny HOCHSCHULE FRESENIUS UNIVERSITY OF APPLIED SCIENCES 40476 D\"USSELDORF, GERMANY.}\\
E-mail adress: sasa.novakovic@hs-fresenius.de\\

\end{document}